\documentclass[12pt]{article} \usepackage{fullpage}
\usepackage{amssymb,amsmath,amsthm}
\usepackage{thmtools,thm-restate}
\usepackage{nicefrac}
\usepackage{url}
\usepackage{color}
\usepackage[american]{babel}
\usepackage{graphicx}
\usepackage{version}
\usepackage{subfigure} % change to subcaption if you'd rather be fancy
\usepackage[textsize=footnotesize]{todonotes}
\usepackage{xspace}

\newif\iffull
\fulltrue
\graphicspath{{figures/}}

\newtheorem{theorem}{Theorem}
\newtheorem{corollary}{Corollary}

\newtheorem{observation}{Observation}

\newtheorem{lemma}{Lemma}

\newcommand{\full}{>}
\usetikzlibrary{decorations.pathmorphing}
\DeclareMathOperator{\overshadows}{\raisebox{-0ex}{\!\begin{tikzpicture}[scale=0.13]
\useasboundingbox (-1,-0.8) rectangle (0.2,1.3);
\draw [solid,sharp corners] (-1,-1) -- (0,-0.3) -- (0,1.5);
\draw [solid,sharp corners] (0.3,-1) -- (0.3,-0.3) -- (0.3,0.7);
\end{tikzpicture}\!
}
}
\DeclareMathOperator{\notovershadows}{\raisebox{-0ex}{\!\begin{tikzpicture}[scale=0.13]
\useasboundingbox (-1,-0.8) rectangle (0.2,1.3);
\draw [solid,sharp corners] (-1,-1) -- (0,-0.3) -- (0,1.5);
\draw [solid,sharp corners] (0.3,-1) -- (0.3,-0.3) -- (0.3,0.7);
\draw [line width=0.6pt,solid] (-0.7,-1.3) -- (0.7,1.5);
\end{tikzpicture}\!\!
}
}

\newcommand{\leaveout}[1]{}

\date{}
\title{Strongly chordal graphs as intersection graphs of trees \\ (Farber's proof revisited)}
\author{Therese Biedl
\thanks{David R.~Cheriton School of Computer
Science, University of Waterloo, Waterloo, Ontario N2L 3G1, Canada.
Supported by NSERC.  
}
}

\begin{document}

\maketitle
\begin{abstract}
In his Ph.D. thesis, Farber proved that every strongly chordal
graph can be represented as intersection graph of subtrees of
a weighted tree, and these subtrees are ``compatible''.   Moreover,
this is an equivalent characterization of strongly chordal graphs.
To my knowledge, Farber never published his results in a conference
or a journal, and the thesis is not available electronically.   As
a service to the community, I therefore reproduce the proof here.

I then answer some questions that naturally arise from the proof.
In particular, the sufficiency proof works by showing the existence
of a simple vertex.   I give here an alternate sufficiency proof 
that directly converts a set of compatible subtrees into a strong elimination 
order.  
%I also show that permitting
%weights for the tree is required, by proving that some strongly chordal
%graphs to do not have a representation via compatible subtrees of an
%unweighted tree.
\end{abstract}

%\linenumbers

%%%%%%%%%%%%%%%%%%%%%%%%%%%%%%%%%%%%%%%%%%%%%%%%%%%%%%%%%%%%%%%%%%%%%%%%
\section{Introduction}

For many graph classes that are closed under taking induced subgraphs, 
there are multiple equivalent ways to characterize the graphs, often in one
of the following forms:
\begin{itemize}
\item Forbidden induced subgraphs.   
\item Existence of a vertex with special properties.
\item Existence of a vertex order with special properties.
\item Representation as intersection graph of some restricted kind of objects.
\end{itemize}
We illustrate this on the example of \emph{chordal graphs}, which
are the graphs where every cycle $C$ of length 4 or more has
a \emph{chord}, i.e., an edge connecting two non-consecutive vertices of $C$.
The forbidden induced subgraphs are hence the cycles of length 4 or more.
One can argue that any chordal graph $G$ (and therefore also every induced
subgraph of $G$) has a \emph{simplicial vertex}, i.e., a vertex $v$ for
which the closed neighbourhood forms a clique.     This in turn immediately
implies that $G$ has a \emph{perfect elimination order} $v_1,\dots,v_n$,
i.e., a vertex order where for $i=1,\dots,n$ vertex $v_i$ is simplicial
in the graph induced by $v_i,\dots,v_n$.       This in turn easily implies
that $G$ is chordal, so all these conditions are equivalent.   Finally,
chordal graphs have a representation as intersection graphs:   Gavril
\cite{Gavril74} showed that a graph
is chordal if and only if it is the ``intersection graph of subtrees of
a tree'' (all terms in quotation marks will be defined formally below).

This note concerns \emph{strongly chordal graphs}, a
subclass of chordal graphs.   Namely, a graph $G$ is strongly chordal if and
only if it is chordal, and every cycle $C$ that has even length at least 6
has a chord between two vertices that have odd distance within $C$ (where distance is 
measured by the number of edges on the path).   Strongly chordal graphs
are of interest because for many problems there are algorithms for strongly
chordal graphs that are significantly more efficient than the fastest
algorithms known for chordal graphs.  See for
example \cite{DahlhausK98,Farber84} for linear-time algorithms for matching and dominating set
in strongly chordal graphs.

Many equivalent definitions are known for strongly chordal graphs
(and have been exploited for designing these algorithms).    Farber
\cite{Farber83} gave a characterization via the existence of a 
``simple vertex'', or the existence of a ``strong elimination order'',
or the absence of a so-called \emph{trampoline} (or \emph{sun}) as induced subgraph.    
Even more characterizations have been developed later \cite{CM14,McKee99}.
But is there a characterization as intersection graphs?

In his Ph.D.~thesis~\cite{FarberThesis}, Farber indeed gave such
a characterization:  strongly chordal graphs are exactly those graphs
that are intersection graphs of a set of subtrees of weighted trees that 
are ``compatible'' in the sense that for any two subtrees, one ``is full''
(or ``overshadows'') the other.   Unfortunately, this characterization
was not included in the papers that he later published \cite{Farber83,Farber84}.
Furthermore, the thesis does not appear to be available in either electronic
or printed form, and neither is the technical report from Simon Fraser University
that (apparently) covers the material.   
%I have not even been able to obtain 
%a printed copy of either thesis or technical report from the university 
%inter-library system.   
Luckily, I was able to get a microfilm version
of the thesis via the {\tt omni} inter-library loan system, though
even this one seems to have been removed now.
For this reason, I decided that as a service to the community it would be helpful to create
an electronic version of Farber's proof.    
So the first part of this note is a proof of the following:

\begin{theorem}[\cite{FarberThesis}]
\label{thm:main}
A graph is strongly chordal if and only if it is the intersection
graph of a compatible collection of subtrees of a rooted weighted tree.
\end{theorem}

The idea of the proof given here is taken directly from Farber's thesis, but I greatly reworded the 
proof and expanded numerous details to make them---in my biased opinion---easier to parse.   
(The original proof is in the appendix.)

Some questions naturally arise from the result and proof.   First, the
``overshadowing'' relationship deserves further explanation; it is not a
partial order but some properties can nevertheless be shown.    Second,
the sufficiency part of the theorem is shown
by arguing that any graph with a compatible tree-representation has a simple 
vertex.    The more common (and algorithmically more useful) characterization
of strongly chordal graph is the existence of strong elimination order.  Getting
a strong elimination order, given the existence of simple vertices, is not immediate (in particular
there are examples where simply enumerating the vertices such that each of them
is simple with respect to the graph of later vertices is \emph{not} a perfect elimination order).
The algorithm that Farber gave for finding a strong elimination order \cite{Farber83}
is not particular straightforward, and not particularly fast.   I therefore give
a different algorithm that reads the strong elimination order directly from a
compatible tree-representation.

\begin{restatable}{theorem}{TreerepPEO}
%\begin{theorem}
\label{thm:treerep_peo}
Given a compatible tree-representation
that defines a graph $G$ with $n$ vertices and $m$ edges, one can find
a strong elimination order of $G$ in $O(n+m)$ time, presuming it is
known for any two adjacent vertices $v,w$ whether tree $T(v)$
overshadows $T(w)$.
%\end{theorem}
\end{restatable}

\iffalse
Second, Farber allowed weights on the arc of the host-tree for his
intersection representation.    This is rather unusual (related
graph classes such as chordal graphs, interval graphs, path graphs
and others are defined on unweighted trees), and one naturally asks
whether weights are required, or whether perhaps one could switch
to unweighted trees by subdividing arcs suitable.   I answer this
question.

\begin{restatable}{theorem}{NotUnweighted}
%\begin{theorem}
\label{thm:not_unweighted}
There are strongly chordal graphs  such that in
any compatible tree representations the host-tree
must have arcs with different weights.
%\end{theorem}
\end{restatable}
\fi

\section{Preliminaries}
\label{sec:preliminaries}

This section clarifies some notation for a graph $G=(V,E)$
(for most standard notation, see for example Diestel \cite{Die12}).
For a vertex $v$, write $N[v]$ for the
\emph{closed neighbourhood} of $v$, i.e., $N[v]=\{w: w{=}v\text{ or } (w,v) \text{ is an edge}\}$.

Let $T$ be a rooted tree (the \emph{host tree}) and let $T_1,\dots,T_n$ be a collection of subtrees of $T$
This defines a graph $G$ (the \emph{intersection graph} of $T_1,\dots,T_n$) that has
vertices $v_1,\dots,v_n$ and an edge $(v_i,v_j)$ if and only if $V(T_i)\cap V(T_j)\neq \emptyset$.
We use the terms ``vertex/edge'' for the elements of $G$, while elements of host-tree $T$
are called ``node/arc'' for ease of distinction.    We describe subtree $T_i$ by its node set,
so we can speak of the sets $T_i\cap T_j$, $T_i\cap T_j$ etc.

Given a graph $G$, a \emph{tree representation}
of $G$ consists of a rooted tree $T$ and a subtree $T(v)$ of $T$
for every vertex of $G$ such that $G$ is the intersection graph of $\{T(v): v\in V\}$.
Given a tree representation of $G$ and a vertex $v\in V$, we use $\mathbf{v}$ for the node
of host-tree $T$ that is the root of subtree $T(v)$.    
%Generally we use bold-face letters for nodes of $T$,  and write $\beta(\mathbf{w})$ 
%for the set of all vertices of $G$ whose subtree includes $\mathbf{w}$.  
Figure~\ref{fig:example}(a-b) illustrates these concepts.

\begin{figure}[ht]
\subfigure[~]{\raisebox{30mm}{\includegraphics[scale=0.7,page=2]{bad_example.pdf}}}
\subfigure[~]{
\begin{minipage}[b]{0.3\linewidth}
%\vspace{0pt}
\begin{tikzpicture}[scale=0.8,level distance=40pt,
level 1/.style={sibling distance=200pt},
level 2/.style={sibling distance=100pt},
level 3/.style={sibling distance=20pt},
level 4/.style={sibling distance=20pt},
level 5/.style={sibling distance=20pt},
edgelabel/.style={draw=none,fill=white,inner sep=0.5,font=\scriptsize},
every node/.style={rounded corners,inner sep=1.5,minimum width=10mm,minimum height=1ex,draw,font=\scriptsize},
label/.style={right,draw=none,font=\scriptsize,color=black},
font = \footnotesize
]
\node (root) {$\mathbf{z}$}
            child {node [] (l3) {$\mathbf{y},z$}
            	child {node [] (l4) {$\mathbf{x},y,z$}
               	    child {node (l5) {$\mathbf{w},x,y,z$}
               	    	child {node [dotted,circle,minimum width=3mm] {}
	               	    child {node [] (l7) {$\mathbf{b},w,x$} }
				child{ edge from parent[draw=none] }
				child{ edge from parent[draw=none] }
		        }
               	    	child {node (l6) [dotted,circle,minimum width=3mm] {}
               	    	    child {node [dotted,circle,minimum width=3mm] (u1) {}
	               	        child {node [] (l8) {$\mathbf{a},w,x$} }
		            }
		        }
               	    	child {node [right=-1mm] {$\mathbf{c},x,y$} }
	            }
	        }
            }
;

\node [draw=none] (rlabel) at ([yshift=-1mm,xshift=60pt]root.south) {};
\draw [draw=none] ([yshift=-1mm]root.south) -- ([yshift=-0pt]rlabel.west) node [midway,above=1ex,label] {
\begin{minipage}{10mm}
$T(w)\notovershadows T(z)$ 
\end{minipage}};
%\draw [thick,dotted,red] ([yshift=-1mm]l3.south) -- ([yshift=-40pt]rlabel.west) node [label] {\begin{minipage}{10mm}
%\end{minipage}};
\draw [draw=none] ([yshift=-1mm]l4.south) -- ([yshift=-80pt]rlabel.west) node [midway,above=1ex,label] {
\begin{minipage}{10mm}
$T(w)\notovershadows T(x)$ 
\end{minipage}};
%\draw [thick,dotted,red] ([yshift=-1mm]l5.south) -- ([yshift=-120pt]rlabel.west) node [label] {\begin{minipage}{10mm}
%\end{minipage}};
\draw [thick,dotted,red] ([yshift=-1mm]l6.south) -- ([yshift=-161pt]rlabel.west) node [midway,below=1ex,label] {\begin{minipage}{10mm}
$\Theta(x,y)$
\end{minipage}};
%\draw [thick,dotted,red] ([yshift=-1mm]l7.south) -- ([yshift=-201pt]rlabel.west) node [label] {\begin{minipage}{10mm}
%\end{minipage}};
\draw [thick,dotted,red] ([yshift=-1mm]l8.south) -- ([yshift=-241pt]rlabel.west) node [midway,below=1ex,label] {\begin{minipage}{10mm}
$\Theta(x,w)$
\end{minipage}};
\node [above right = -1mm and -4mm of u1,draw=none] {$\mathbf{u_1}$};
\node [above right = -1mm and -4mm of l6,draw=none] {$\mathbf{u_2}$};
\end{tikzpicture}
\end{minipage}
} % subfigure
\hspace*{\fill}
\subfigure[~]{\includegraphics[scale=0.89,page=1]{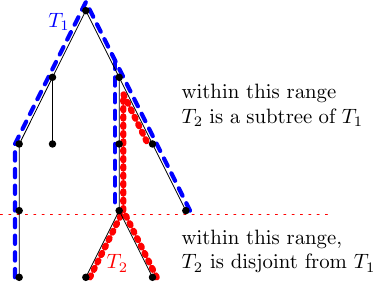}}
\hspace*{\fill}
\caption[A chordal graph with a tree representation]{(a) A chordal graph (actually strongly chordal) with a perfect elimination order (actually a strong
elimination order). (b) The tree-representation that we get when applying Lemma~\ref{lem:order_treerep}.
Symbol 
\begin{tikzpicture}\node[rounded corners,inner sep=1.5,minimum width=10mm,draw,font=\scriptsize]{$\mathbf{x},y,z$}; \end{tikzpicture} 
means that this is node $\mathbf{x}$ and it belongs to $T(x),T(y)$ and $T(z)$.
Dotted lines indicate cutoff values for some edges; we also indicate some witnesses for not overshadowing.
(c) Abstract illustration of the concept of overshadowing.
}
\label{fig:example}
\end{figure}

Gavril proved in 1974 that a graph is an
intersection graph of subtrees of a tree if and only if it is chordal \cite{Gavril74}.
Farber in his thesis \cite{FarberThesis} gave a different proof of this result
that reads the tree representation directly from a perfect elimination order (we
review parts of it in the appendix).   The following is obvious but will be needed below:

\begin{observation}
\label{obs:intersection}
\label{obs:intersect}
Let $\{T(v)\}$ be a tree-representation of a graph $G$, and let $v,w$ be two vertices with $v\in N[w]$.    Then $\mathbf{v}\in T(w)$ if and only if $d(\mathbf{v})\geq d(\mathbf{w})$.
\end{observation}
\begin{proof}
Since $v\in N[w]$, we have a node $\mathbf{x}\in T(v)\cap T(w)$.   By
replacing $\mathbf{x}$ by its parent, if needed, we may assume that $\mathbf{x}$
either has no parent, or the parent of $\mathbf{x}$ is not in $T(v)\cap T(w)$.
But then $\mathbf{x}$ is the root of one of $T(v)$ and $T(w)$, and the other root is an ancestor of $\mathbf{x}$, hence cannot have greater depth.
\end{proof}

\section{Representations of strongly chordal graphs}

Recall that a strongly chordal graph is a chordal graph where every
even-length cycle $C$ of length at least 6 has a chord whose endpoints have odd distance along $C$.
Two equivalent characterizations of strongly chordal graphs were proved by
Farber \cite{Farber83}.
The first one says says that a graph is strongly chordal if and only if it has
a strong elimination order.    
\begin{quotation}
A \emph{strong elimination order} of a graph $G=(V,E)$ is an ordering  $v_1,\dots,v_n$ of $V$
with the property that for each $i,j,k$ and $\ell$, if $i<j$, $k<\ell$, $v_k,v_\ell\in N[v_i]$
and $v_k\in N[v_j]$ then $v_\ell\in N[v_j]$.
\end{quotation}
The second equivalent characterization of strongly chordal graphs uses the concept of a simple vertex.
\begin{quotation}
A vertex $v$ is called \emph{simple} if the closed neighbourhood $N[v]$  of $v$
can be ordered as $u_1,\dots,u_d$ such that $N[u_1]\subseteq \dots \subseteq N[u_d]$.   
\end{quotation}
Farber showed that a graph $G$ is strongly chordal if and only if every induced subgraph of $G$ has a simple vertex \cite{Farber83}.
We now include an intersection representation in these equivalences, but need some definitions first.

Let $\{T(v)\}$ be a tree-representation of a graph $G$, with host tree $T$.
We permit $T$ to have \emph{arc-weights}, i.e., an assignment of positive integer weights
to the arcs of $T$.    (The dotted nodes in Figure~\ref{fig:example}(b) are
added to illustrate non-unit weights.)
Define the \emph{depth} $d(\mathbf{w})$ of a node $\mathbf{w}$ of $T$ to be the
weighted distance from the root, i.e., 
the sum of the edge-lengths on the unique path from $\mathbf{w}$ to the root of $T$.
In what follows, terms such as ``depth'' and ``farther away from the root'' always
refer to this \emph{weighted} version.

Let $T_1$ and $T_2$ be two subtrees of $T$.   We say that
$T_1$ \emph{overshadows} $T_2$,
written $T_1\overshadows T_2$,%
\footnote{Both name and symbol deviate from Farber's thesis; see Section~\ref{sec:partial_order} for further discussion.}
if all nodes of $T_2\setminus T_1$ are strictly farther away from the root than all nodes in $T_2\cap T_1$.
See also Figure~\ref{fig:example}(b-c). 
Clearly two disjoint trees overshadow each other, so the condition is non-trivial only if $T_1$ and $T_2$ have nodes in common. 
We can then give two equivalent statements that we find helpful in illustrations and arguments.
Define $\Theta$ (the \emph{cutoff value}) to be the maximum distance of a node in $T_1\cap T_2$.
Then $T_1\overshadows T_2$ (for non-disjoint trees) if and only if, restricted to the nodes of depth at most $\Theta$,
tree $T_2$ is a subset of $T_1$, while, restricted to the nodes of depth more than $\Theta$, tree $T_2$ is disjoint from $T_1$.
It is perhaps helpful to state when a tree $T_1$ does \emph{not} overshadow tree $T_2$
(which we denote by $T_1\notovershadows T_2$):   This happens if and only if there is
a \emph{witness-node} $\mathbf{y} \in T_2\setminus T_1$ for which the depth is at most the cutoff value.

Two trees $T_1,T_2$ are called \emph{compatible}
if $T_1\overshadows  T_2$ or $T_2\overshadows T_1$ (or both),
and a tree-repre\-sen\-tation $\{T(v)\}$ is called \emph{compatible}
if all pairs of trees are compatible.
For example the subtrees in Figure~\ref{fig:example}(b) are compatible.

With this, all the ingredients for Theorem~\ref{thm:main} have 
been defined, and we can now review its proof.    This is done
in a cycle of implication, where the third one
(the existence of simple vertices implies the existence
of a strong elimination order) can be found in \cite{Farber83}.
We state the other two implications
as separate lemmas so we can emphasize (under
``Furthermore'') some side-effects of the proof that may be of interest.

\begin{lemma}
\label{lem:order_treerep}
\label{lem:ordering_treerep}
If $G$ has a strong elimination order $v_1,\dots,v_n$,
then $G$ has a compatible tree-representation $\{T(v)\}$.

Furthermore, the host tree has exactly $n$ nodes, corresponding
to the roots $\mathbf{v_1},\dots,\mathbf{v_n}$ of the
trees representing the vertices, and $\mathbf{v_j}$ has depth $n-j$
for all $j\in \{1,\dots,n\}$.
The representation can be found in linear time.
\end{lemma}
\begin{proof} 
Construct a tree $T$ and subtrees $T_j$ for $v_j\in V$ as follows:
\begin{itemize}
\item Node $\mathbf{v_n}$ is the root of the tree.
\item For $j=n{-}1,\dots,2,1$, let $v_k$ be the \emph{first strict successor} of $v_j$,
	i.e., let $k$ be minimal such that $k>j$ and $v_k\in N[v_j]$.
	(If there is no such neighbour, then define $k:=n$.)
	Make node $\mathbf{v_j}$ the child of $\mathbf{v_k}$, and give 
	weight $k-j$ to this arc.    Using induction one
	then sees that $\mathbf{v_j}$  has depth $n-j$.
\item For $k=1,\dots,n$, let $T_k$ by the subgraph of $T$ induced by the 
	nodes of the \emph{predecessors} of $v_k$, i.e., $T_k=\{\mathbf{v_i}: i\leq k \text{ and } v_i\in N[v_k]\}$.
	Note that this includes $\mathbf{v_k}$, and (by definition of a predecessor) $ \mathbf{v_k}$ has the smallest depth among the nodes in $T_k$, so is its root.
\end{itemize}

The tree in Figure~\ref{fig:example}(b) has been constructed with this
method, using the vertex order from Figure~\ref{fig:example}(a).
Clearly this representation can be found in $O(n+m)$ time if $G$
has $m$ edges:   Tree $T(v_j)$ contains at most $\deg(v_j)$ nodes,
and they can be read directly from the neighbourhood of $v_j$ by
comparing indices.

We now must show three things:   (1) Each node-set $T_\ell$ actually
forms a subtree, i.e., induces a connected set;
(2) $G$ is the intersection graph of $\{T_1,\dots,T_n\}$;
and (3) the trees are compatible.     
Claims (1) and (2) are straightforward for any perfect elimination order,
but will be repeated here for completeness.    

To see (1), it suffices
to show that for any predecessor $v_i$ of $v_\ell$,
the entire path from $\mathbf{v_i}$ to $\mathbf{v_\ell}$ in $T$ belongs to $T_\ell$.
This is obvious if $i=\ell$, so assume not.  We then have $i<\ell\leq n$ by choice of $T_\ell$, so $\mathbf{v_i}$ has a parent, say $\mathbf{v_k}$.
Since $v_\ell$ is a strict successor of $v_i$, so must be $v_k$, 
and $k\leq \ell$ by choice of parent of $\mathbf{v_i}$.
By $i<k$ we have $d(\mathbf{v_i})>d(\mathbf{v_k})$, hence $\mathbf{v_i}\in T_k$ by Observation~\ref{obs:intersect}.   So $\mathbf{v_i}$ is common to
$T_k$ and $T_{\ell}$, which gives $v_k\in N[v_\ell]$ and by $k\leq \ell$ 
therefore $\mathbf{v_k}\in T_\ell$.
By induction on the depth, the path from $\mathbf{v_k}$ to $\mathbf{v_\ell}$
is in $T_\ell$, and therefore so is the one from $\mathbf{v_i}$.

To see (2), observe that if $v_k\in N[v_\ell]$, where (say) $k\leq \ell$, 
then $\mathbf{v_k}\in T_\ell$, and since $\mathbf{v_k}\in T_k$ therefore
the two trees intersect.   Vice
versa, if some node (say $\mathbf{v_i}$) is in $T_k \cap T_\ell$ for some
$k<\ell$, then
$v_i$ is a predecessor of both $v_k$ and $v_\ell$.   We then must have
edge $(v_k,v_\ell)$ by applying the condition of a strong elimination order
with $j{=}k$.

To show (3), we will show something
stronger, namely, that $T_\ell \overshadows T_k$ for any
$\ell>k$.    This clearly holds if the trees are disjoint,
so assume that $T_\ell$ and $T_k$ share nodes, and let
$\mathbf{v_i}$ be a node in $T_\ell\cap T_k$ with maximal depth.
We claim that $D:=n-i$ works as cutoff value for relationship
$T_\ell \overshadows T_k$.
By definition, no node of depth greater than $D$ belongs to $T_\ell\cap T_k$,
so we only have to show that no node of $T_k\setminus T_\ell$ has depth at most $D$.
Assume for contradiction that there is a node $\mathbf{v_j}\in T_k\setminus T_\ell$
with $d(\mathbf{v_j})\leq D$.   We know $d(\mathbf{v_j})=n-j$ while $D=n-i$, so
$j\geq i$, and actually $j>i$ since $\mathbf{v_i}\in T_\ell$.    
We also know that $j\leq k$ since $\mathbf{v_j}\in T_k$, 
so $v_j$ is a predecessor of $v_k$ (possibly $j=k$).    In summary, we have $i<j$, $k<l$,
edges $(v_i,v_k)$, $(v_i,v_\ell)$ and $v_j\in N[v_k]$.   Since we have
a strong elimination order therefore $v_j\in N[v_\ell]$, and by $j\leq k<\ell$
it is a predecessor of $v_\ell$.    This contradicts $\mathbf{v_j}\not\in T_\ell$.
\end{proof}

For any graph that has a compatible tree-representation, any induced subgraph
also has a compatible tree-representation.   Therefore for the sufficiency
of Theorem~\ref{thm:main}, it suffices to show the following.

\begin{lemma}
\label{lem:treerep_simple}
Let $G$ be a graph that has a compatible tree-representation $\{T(v)\}$. 
Then $G$ has a simple vertex. 

Furthermore, any vertex $v$ that maximizes the depth of the root $\mathbf{v}$ of $T(v)$ is a simple vertex.
\end{lemma}
\begin{proof}
Let $u_1,u_2$ be two vertices in $N[v]$.   Since we have a compatible tree-representation,
we have (up to renaming) that $T(u_2)\overshadows T(u_1)$.   We will show that this implies
$N[u_1]\subseteq N[u_2]$, and since the argument can be applied to any two vertices of $N[v]$,
the neighbourhood of $v$ can therefore be ordered suitably.

%Assume for contradiction that there exists a vertex $x\in N[u_1]\setminus N[u_2]$.   By Observation~\ref{obs:intersection} then $\mathbf{y}\in T(x)\cap T(u_1)$ for some $\mathbf{y}\in \{\mathbf{x},\mathbf{u_1}\}$.
%We also know that $\mathbf{y}\not\in T(u_2)$
%since $(x,u_2)$ is not an edge,  so $\mathbf{y}\in T(u_1)\setminus T(u_2)$.

From~Observation~\ref{obs:intersection}, we know that $\mathbf{v}\in T(u_i)$ or $\mathbf{u_i}\in T(v)$
for $i=1,2$.   But by choice of $v$ we have $d(\mathbf{v})\geq d(\mathbf{u_i})$, meaning that
$\mathbf{u_i}\in T(v)$ implies $\mathbf{u_i}=\mathbf{v}$.   So in both cases we have $\mathbf{v}\in T(u_i)$,
which implies that $T(u_1)\cap T(u_2)$ includes node $\mathbf{v}$ and $(u_1,u_2)$ is an edge.
Furthermore, the cutoff-value for $T(u_2)\overshadows T(u_1)$ is at least $d(\mathbf{v})$.

Now assume for contradiction that there exists a vertex $x\in N[u_1]\setminus N[u_2]$.
We have an edge $(x,u_1)$, so
by~Observation~\ref{obs:intersection} we have a node $\mathbf{y}\in T(x)\cap T(u_1)$
with $\mathbf{y}\in \{\mathbf{x},\mathbf{u_1}\}$.   
We also know that $\mathbf{y}\not\in T(u_2)$
since $\mathbf{y}\in T(x)$ and $x\not\in N[u_2]$.
Therefore $\mathbf{y}\in T(u_1)\setminus T(u_2)$,
which by the cutoff-value for $T(u_2)\overshadows T(u_1)$ implies $d(\mathbf{y})>d(\mathbf{v})$.
But $\mathbf{y}$ is the root of a subtree (either $T(x)$ or $T(u_1)$), 
which contradicts the choice of $v$.
\end{proof}

%%%%%%%%%%%%%%%%%%%%%%%%%%%%%%%%%%%%%%%%%%%%%%%%%%%%%%%%%%%%%%%%%%%%%%%%
\section{Discussion items}

\subsection{The $\overshadows$ relationship}
\label{sec:partial_order}

As mentioned earlier, the name ``$T_1$ overshadows $T_2$'' for the relationship was
newly introduced here; Farber called this ``$T_1$ is \emph{full}
with respect to $T_2$'', which I felt was not descriptive enough.
More importantly, Farber use ``$T_1>T_2$'' to denote this relationship,
which I felt was inappropriate since the relationship is not
necessarily a partial order.
In particular, it need not be transitive since for any two disjoint trees $T_1,T_2$
we have $T_1\overshadows T_2$; if we now choose another tree $T_3$
that intersects $T_1$ but not $T_2$ then we also have $T_2\overshadows T_3$
but it may or may not be true that $T_1\overshadows T_3$.   We can also
have circular relationships, even for intersecting trees, see Figure~\ref{fig:partial_order}.

\begin{figure}[ht]
\hspace*{\fill}
\subfigure[~]{\begin{tikzpicture}[scale=1.0,level distance=40pt,
edgelabel/.style={draw=none,fill=white,inner sep=0.5,font=\scriptsize},
every node/.style={rounded corners,inner sep=1.5,minimum width=10mm,minimum height=1ex,draw,font=\scriptsize},
label/.style={right,draw=none,font=\tiny,color=black},
font = \footnotesize
]

\node (root) {$\mathbf{3}$}
	child { node [] {$\mathbf{1},3$}}
	child { node [] {$\mathbf{2}$}}
;
\end{tikzpicture}
}
\hspace*{\fill}
\subfigure[~]{\begin{tikzpicture}[scale=1.0,level distance=40pt,
every node/.style={rounded corners,inner sep=1.5,minimum width=10mm,minimum height=1ex,draw,font=\scriptsize},
font = \footnotesize
]

\node (root) {$\mathbf{1},\mathbf{2},\mathbf{3}$}
	child { node [] {$1$}}
	child { node [] {$2$}}
	child { node [] {$3$}}
;
\end{tikzpicture}
}
\hspace*{\fill}
\subfigure[~]{\begin{tikzpicture}[scale=1.0,level distance=20pt,
every node/.style={rounded corners,inner sep=1.5,minimum width=10mm,minimum height=1.5ex,draw,font=\scriptsize},
font = \footnotesize
]

\node (root) (r) {~~}
	child { 
		child {
			node [label={above left:$\mathbf{z}{=}\mathbf{z_0}$}] (z0) {$0,1,2,\dots$}
			edge from parent [draw=none];
		}
		edge from parent [draw=none];
	}
	child { 
		child { 
			child {
				node [label={above right:$\mathbf{z_1}$}] (z1) {$1,2$}
				edge from parent [draw=none];
			}
			edge from parent [draw=none];
		}
		edge from parent [draw=none];
	}
	child { node [label={above:$\mathbf{y_0}$}] {$\neg 0, 1$} edge from parent[decorate, decoration=snake]}
;
\draw [decorate, decoration=snake] (r.south west) -- (z0.north);
\draw [decorate, decoration=snake] (r.south) -- (z1.north);
\end{tikzpicture}
}
\hspace*{\fill}

\caption{(a) Three compatible subtrees where $T_1\overshadows T_2 \overshadows T_3$, but $T_1\notovershadows T_3$.
(b) Three compatible intersecting subtrees $T_1\overshadows T_2 \overshadows T_3 \overshadows T_1$.
(c) For the proof of Lemma~\ref{lem:not_cyclic}.
}
\label{fig:partial_order}
\end{figure}
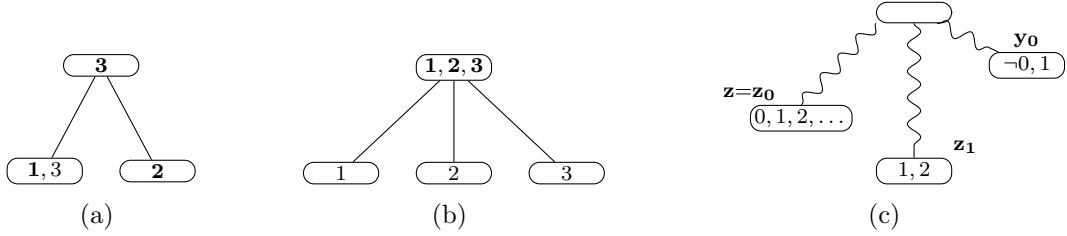

However, we can exclude a circular relationship,
at least in a compatible set of trees,
 if we look at the ``does not overshadow'' relationship, 

\begin{lemma}
\label{lem:not_cyclic}
Let $\{T_1,\dots,T_k{=}T_0\}$ be a set of $k\geq 2$ subtrees of a tree $T$ where
$T_0\notovershadows T_1$, $T_1 \notovershadows T_2$, \dots, $T_{k-1} \notovershadows T_k{=}T_0$.
Then the trees are not compatible.
\end{lemma}
\begin{proof}
For $i=0,\dots,k{-}1$, since $T_i\notovershadows T_{i+1}$, there must exist
a witness-node, i.e., a node $\mathbf{y_i}\in T_{i+1}\setminus T_i$ that is no farther from the root
than some node $\mathbf{z_i}\in T_{i+1}\cap T_i$.   Let $\mathbf{z}$ be
the node in $\{\mathbf{z_0},\dots,\mathbf{z_{k-1}}\}$ that is closest to the root;
up to renaming we may assume that $\mathbf{z}=\mathbf{z_0}$.

If $T_2\notovershadows T_1$ then the two trees are not compatible and we are done.
So assume that $T_2\overshadows T_1$ and consider Figure~\ref{fig:partial_order}(c).    
We know that $d(\mathbf{z_1})\geq d(\mathbf{z})$ by choice of $\mathbf{z}$, 
and therefore $\mathbf{z}\in T_1$ must also be in $T_2$.
Repeating the argument one sees that in fact $\mathbf{z}$ belongs to \emph{all} of $T_1,\dots,T_k$.

Recall that $\mathbf{y_0}$ is no farther from the root than $\mathbf{z}$ and belongs to $T_1$
but not to $T_k{=}T_0$.   We therefore have an index $i$ with $1\leq i<k$
such that $\mathbf{y_0}\in T_i$ but $\mathbf{y_0}\not\in T_{i+1}$.    Since
$\mathbf{z}\in T_i\cap T_{i+1}$, %and $d(\mathbf{z})=d(\mathbf{z_0})\geq d(\mathbf{y_0})$,
therefore $T_{i+1}\notovershadows T_i$,
which shows that $T_i$ and $T_{i+1}$ are not compatible.
\end{proof}

\subsection{Obtaining strong elimination orders}

Of the many equivalent descriptions of strongly chordal graphs, the one most
useful for algorithm design appears to be the strong elimination order, which
was used for example for algorithms for matching and dominating set \cite{DahlhausK98,Farber84}. 
So a natural question is how one can obtain a strong elimination order,
given a compatible tree-representation.    Lemma~\ref{lem:treerep_simple}
shows how to get a simple vertex:  Take one whose subtree has the deepest root.
However, this does not given an ``obvious'' way to obtain a strong elimination
order.   In particular, if we sort vertices
by decreasing depth of their roots 
(obtaining what we call a \emph{bottom-up enumeration order}),
then each vertex is simple w.r.t.~the subgraph formed by the later vertices,
but it does not necessarily give a strong elimination order
(see the example in Figure~\ref{fig:example2}).
%\begin{observation}
%\label{obs:bottom_up_bad}
%There exists a compatible tree-representation such that the
%bottom-up enumeration order, with ties broken unsuitably,
%is \emph{not} a strong elimination order.
%\end{observation}
%\begin{proof}
%To see an example, consider the tree-representation in Figure~\ref{fig:example2}; one easily verifies
%that this is compatible with $T(k)\overshadows T(\ell)$.   One possible bottom-up enumeration order is $i,j,k,\ell$.
%We have edges $(i,k),(i,\ell)$ and $(j,k)$, but we do \emph{not} have edge $(j,\ell)$, so this is not
%a strong elimination order.
%\end{proof}

\begin{figure}[ht]
\hspace*{\fill}
\begin{tikzpicture}[scale=1.0,level distance=40pt,
edgelabel/.style={draw=none,fill=white,inner sep=0.5,font=\scriptsize},
every node/.style={rounded corners,inner sep=1.5,minimum width=10mm,minimum height=1ex,draw,font=\scriptsize},
label/.style={right,draw=none,font=\tiny,color=black},
font = \footnotesize
]

\node (root) {$\mathbf{k}, \mathbf{\ell}$}
	child {edge from parent[draw=none]}
	child { node [draw=none] {}
            child {node (l3) {$\mathbf{i},k,\ell$} }
		edge from parent [draw = none];
	}
	child { node [] {$\mathbf{j},k$}}
;
\draw (root.south) -- (l3.north);
\end{tikzpicture}
\hspace*{\fill}
\caption{A tree representation 
that is compatible (with $T(k)\overshadows T(\ell)$).   One possible bottom-up enumeration order is $i,j,k,\ell$,
which is \emph{not} a strong elimination order since there are
edges $(i,k),(i,\ell)$ and $(j,k)$, but no edge $(j,\ell)$.}
\label{fig:example2}
\label{fig:bottom_up_bad}
\end{figure}
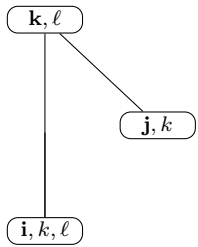

We can of course get a strong elimination order by applying the (algorithmic) proof
by Farber \cite{Farber83} that finds such an order as long as every induced subgraph
has a simple vertex.  However,
Farber's algorithm is somewhat complicated and involves ordering the vertices by
the subset-relation of their neighbourhoods while changing the graph.   The run-time
has not been analyzed, but is at least quadratic and possibly worse.   We give here
instead an algorithm that reads a strong elimination order directly from the
tree representation, and works in linear time if know for any edge $(v,w)$ 
which of $T(v),T(w)$ overshadows the other.   We need some helper-results.

\begin{observation}
\label{obs:compatible}
\label{obs:overshadow}
If $(v,w)$ is an edge and $T(v) \overshadows T(w)$,
then $d(\mathbf{v})\leq d(\mathbf{w})$ (and in
particular therefore $\mathbf{w}\in T(v)$).
%Let $\{T(v)\}$ be a compatible tree representation.
%If $(v,w)$ is an edge and $T(v) \overshadows T(w)$,
%then $\mathbf{w}\in T(v)$.
\end{observation}
\begin{proof}
If we had $d(\mathbf{v})>d(\mathbf{w})$, then $\mathbf{v}$
(which is in $T(w)$ by Observation~\ref{obs:intersection})
would be a strict descendant of $\mathbf{w}$.   But
then $\mathbf{w}\in T(w)\setminus T(v)$ has smaller depth
than $\mathbf{v}\in T(w)\cap T(v)$, which implies
$T(v)\notovershadows T(w)$.
\end{proof}

\begin{lemma}
\label{lem:bottom_up_good}
Let $\{T(v)\}$ be a tree-representation, and let $v_1,\dots,v_n$
be an enumeration of the vertices such that $T(v_j)\overshadows T(v_i)$ for all $i<j$.
%that satisfies the following:
%\begin{itemize}
%\item $d(\mathbf{v_1})\geq \dots \geq d(\mathbf{v_n})$,
%\item if $d(\mathbf{v_\ell})=d(\mathbf{v_k})$ for some $\ell>k$, then $T(v_\ell)\overshadows T(v_k)$.
%\end{itemize}
Then $v_1,\dots,v_n$ is a strong elimination order.
\end{lemma}

Before proving this, we note that such an order is always a bottom-up
enumeration order by Observation~\ref{obs:overshadow}.  But the reverse is not true;
for example in Figure~\ref{fig:bottom_up_bad} we have bottom-up enumeration order $i,j,k,\ell$,
but $T(\ell) \notovershadows T(k)$.

\begin{proof}
For ease of writing, we use $T_i$ rather than $T(v_i)$ for all $i$.
To show the condition for a strong elimination order,
fix arbitrary indices $i<j$ and $k<\ell$ with $v_i,v_j\in N[v_k]$ and $v_i\in N[v_\ell]$;
we have to show that $v_j\in N[v_\ell]$.
Note that the premise is symmetric in $\{i,j\}$ and $\{k,\ell\}$, so up to renaming
we have $i\leq k$.   By Observation~\ref{obs:intersection} and $d(\mathbf{v_i})
\geq d(\mathbf{v_k}) \geq d(\mathbf{v_\ell})$ then $\mathbf{v_i}\in T_k\cap T_\ell$;
in particular therefore $(v_k,v_\ell)$ is an edge and $\mathbf{v_k}\in T_\ell$.

If $j\geq k$ then $v_j\in N[v_k]$ implies $\mathbf{v_k}\in T_j$; since $\mathbf{v_k}\in T_\ell$
then trees $T_j$ and $T_\ell$ intersect and $v_j\in N[v_\ell]$ as desired.   So we are done
unless $j<k$, which implies $\mathbf{v_j}\in T_k$.
By assumption we have $T_\ell \overshadows T_k$, and so all nodes in $T_k\setminus T_\ell$
must have depth strictly greater than $\mathbf{v_i}\in T_k\cap T_\ell$.     Since $d(\mathbf{v_j})\leq d(\mathbf{v_i})$
and $\mathbf{v_j}\in T_k$, therefore $\mathbf{v_j}\in T_\ell$ and $v_j\in N[v_\ell]$ as desired.
\end{proof}

Observe
that $d(\mathbf{v})>d(\mathbf{w})$ immediately implies that $T(w)\notovershadows T(v)$
by Observation~\ref{obs:overshadow}, 
as does $d(\mathbf{v})=d(\mathbf{w})$ if the nodes $\mathbf{v},\mathbf{w}$
are distinct (and therefore the tree $T(v),T(w)$ are disjoint).    
Therefore we have to break ties by overshadowing only between
vertex-pairs $v,w$ with $\mathbf{v}=\mathbf{w}$.   In particular, if (as in
the tree representation obtained via Lemma~\ref{lem:ordering_treerep}) all
vertices have distinct roots of subtrees, then any bottom-up enumeration order is
a strong elimination order.

Before turning Lemma~\ref{lem:bottom_up_good}
into an algorithm to find the order, we first
show a result that was claimed in~\cite{BG-CCCG24}, but
the proof (erroneously) assumed that a strong elimination order can be
obtained by lining up simple vertices. 
An \emph{RDV graph} is a graph with a tree-representation $\{T(v)\}$ where 
every subtree has exactly one leaf.   In other words, it is the intersection
graph of downward paths in a rooted tree.    

\begin{corollary}
Let $G$ be an RDV-graph with a tree-representation $\{T(v)\}$ where every
subtree is a downward path in the host tree.    Then any bottom-up enumeration order
$v_1,\dots,v_n$ is a strong elimination order.
\end{corollary}
\begin{proof}
We will show that for any two vertices $v,w$ with $d(\mathbf{v})\leq d(\mathbf{w})$
tree $T(v)$ overshadows $T(w)$; this shows that the tree-representation is compatible
and the result then follows from Lemma~\ref{lem:bottom_up_good}.
Clearly we have $T(v)\overshadows T(w)$ if the trees are disjoint.
So assume they share nodes, which means $\mathbf{w}\in T(v)$ by Observation~\ref{obs:intersect}.
Since both subtrees are downward paths,
they go downward from $\mathbf{w}$ along the same path until some node $\mathbf{z}$,
whence they diverge and the rest of the paths are disjoint.   Setting
the cut-off point to be $d(\mathbf{z})$, therefore  $T(v)\overshadows T(w)$.
\end{proof}

To turn Lemma~\ref{lem:bottom_up_good} into an algorithm (and hence
prove Theorem~\ref{thm:treerep_peo}), we only
have to describe how to quickly find a suitable order.   

\TreerepPEO*
\begin{proof} %(of Theorem~\ref{thm:treerep_peo})
Define an auxiliary directed graph $H$ with the same vertices as $G$
and the following directed edges:    For any edge $(v,w)$ of $G$,
add a directed graph $v\rightarrow w$ to $H$ if $T(v)\notovershadows T(w)$,
i.e., if we \emph{must} put $v$ earlier in the order.    By
Lemma~\ref{lem:not_cyclic}, a directed cycle in $H$ means that the
tree-representation was not compatible, and we can detect this in 
$O(m+n)$ time since $H$ has $O(m)$ edges.    So we are done if $H$
has a directed cycle.   Otherwise, find a topological order $v_1,\dots,v_n$
of $H$, which can be done in linear time.    For any $v_i,v_j$ with $i<j$, 
if there is no edge $(v_i,v_j)$ 
in $G$ then the trees $T(v_i),T(v_j)$ are disjoint and therefore $T(v_j)\overshadows T(v_i)$.
If there is an edge $(v_i,v_j)$
in $G$, then by choice of the topological order we do \emph{not} have a 
directed edge $v_j\rightarrow v_i$ in $H$.   Therefore $T(v_j)\overshadows T(v_i)$ by definition of $H$,
and $v_1,\dots,v_n$ is a strong elimination order by Lemma~\ref{lem:bottom_up_good}.
\end{proof}

\subsection{Unweighted tree representations?}

Let us end the paper with an open question.    Farber's tree representation
uses a \emph{weighted} host tree, which is rather unusual.   Are the weights
ever necessary, i.e., is there a strongly chordal graph that does not
have a compatible tree representation with uniform arc weights?    

Note that we may assume that all weights are integers by the proof of Lemma~\ref{lem:order_treerep},
and it would be easy to replace an edge 
of weight $k>1$ by a path with $k-1$ new subdivision-nodes,    see the dotted nodes
in Figure~\ref{fig:example}.  This maintains a tree
representation if we add the new nodes to trees suitably, but can we always
achieve a compatible tree representation?   

To see that this is non-trivial, consider subdivision-nodes $\mathbf{u_1}$ and $\mathbf{u_2}$
in Figure~\ref{fig:example}.
Clearly they must get added to $T(x)$ and $T(w)$, since these trees include both the
upper and the lower end of the arc that got subdivided by $\mathbf{u_1}$ and $\mathbf{u_2}$.
But should we also add them to $T(y), T(z)$ or $T(a)$?

One can observe that $\mathbf{u_2}$ \emph{must} be added to $T(y)$,
because we have $T(x)\notovershadows T(y)$ due to node $\mathbf{y}$,
so we must maintain that $T(y)\overshadows T(x)$,
and $\mathbf{u_1}\in T(x)$ has the same depth as $\mathbf{c}\in T(x)\cap T(y)$.    
So some of the subtrees that include the upper end must be extended
into the new nodes, even if they do not include the lower end.

On the other hand,
we \emph{must not} add $\mathbf{u_1}$ to $T(z)$,
because we have $T(w)\notovershadows T(z)$ to due to node $\mathbf{z}$, 
so we must maintain that $T(z)\overshadows T(w)$,
and $\mathbf{b}\in T(w)\setminus T(z)$  has the same depth as $\mathbf{u_1}\in T(w)$.
So some of the subtrees that include the upper end \emph{must not} be extended
into the new nodes.

So neither of the two obvious strategies to extend subtrees into new nodes works.
In this specific example one can find a suitable assignment, 
and there is also the option of rearranging the host tree entirely.
%---deleting the top three nodes would maintain a tree
%representation and free up many more ways in which subtrees could be kept
%compatible).    
So it may well be true that all strongly chordal graphs
have compatible tree-representations with unit arc lengths, but this remains open.

%%%%%%%%%%%%%%%%%%%%%%%%%%%%%%%%%%%%%%%%%%%%%%%%%%%%%%%%%%%%%%%%%%%%%%%%
\bibliographystyle{plain}
\bibliography{full,papers}

\begin{thebibliography}{10}

\bibitem{BG-CCCG24}
T.~Biedl and P.~Gokhale\student{}.
\newblock Finding maximum matchings in {RDV} graphs efficiently.
\newblock In {\em Canadian Conference on Computational Geometry (CCCG 2024)},
  pages 305--312, 2024.

\bibitem{CM14}
Pablo~De Caria and Terry~A. McKee.
\newblock Maxclique and unit disk characterizations of strongly chordal graphs.
\newblock {\em Discuss. Math. Graph Theory}, 34(3):593--602, 2014.

\bibitem{DahlhausK98}
E.~Dahlhaus and M.~Karpinski.
\newblock Matching and multidimensional matching in chordal and strongly
  chordal graphs.
\newblock {\em Discrete Applied Mathematics}, 84(1):79--91, 1998.

\bibitem{Die12}
R.~Diestel.
\newblock {\em Graph Theory, 4th Edition}, volume 173 of {\em Graduate texts in
  mathematics}.
\newblock Springer, 2012.

\bibitem{FarberThesis}
M.~Farber.
\newblock {\em Applications of {L.P.} duality to problems involving
  independence and domination}.
\newblock PhD thesis, Rutgers University, 1982.
\newblock Also issued as Technical Report 81-13, Computing Science Department,
  Simon Fraser University, 1981.

\bibitem{Farber83}
M.~Farber.
\newblock Characterizations of strongly chordal graphs.
\newblock {\em Discret. Math.}, 43(2-3):173--189, 1983.

\bibitem{Farber84}
M.~Farber.
\newblock Domination, independent domination, and duality in strongly chordal
  graphs.
\newblock {\em Discrete Appl. Math}, 7:115--130, 1984.

\bibitem{Gav74}
F.~Gavril.
\newblock Algorithms on circular-arc graphs.
\newblock {\em Networks}, 4:357--369, 1974.

\bibitem{Gavril74}
F.~Gavril.
\newblock The intersection graphs of subtrees in trees are exactly the chordal
  graphs.
\newblock {\em Journal of Combinatorial Theory Ser. B}, 16:47--56, 1974.

\bibitem{McKee99}
Terry~A. McKee.
\newblock A new characterization of strongly chordal graphs.
\newblock {\em Discret. Math.}, 205(1-3):245--247, 1999.

\end{thebibliography}

\begin{appendix}
\section{The original proof}

This appendix gives the original proof from Farber's thesis,
taken verbatim except where indicated with ``$\langle$ \hspace*{3mm} $\rangle$''.
As mentioned earlier, the (typewritten, then microfilmed, then
retrieved from microfilm) version that I had access to was of somewhat lower
quality.  In particular it was often difficult to
tell indices such as $1$, $i$ and $j$ apart, and $<$ and $>$
were often visible only as a diagonal and so hard to differentiate.
Where the scan was illegible I tried my best to guess from context
of what was likely meant and I apologize in advance
for inadvertently introduced errors.

\subsection{Preliminaries}

$\langle$I list here a few things that precede the theorem in the thesis,
but are needed in its proof.$\rangle$

\begin{itemize}
\item \mbox{}$\langle$p.~10$\rangle$ If $u,v\in V(G)$ we will write $u\sim v$ if
	$u$ and $v$ are equal or adjacent.

\item \mbox{}$\langle$p.~37$\rangle$ {\bf Definition:} A \emph{strong elimination ordering} of $G$ is an ordering $v_1,\dots,v_n$
	of $V(G)$ satisfying the following two conditions for each $i,j,k$ and $\ell$:
	\begin{enumerate}
	\item If $i>j>k$ and $[v_i,v_k], [v_j,v_k]\in E(G)$ then $[v_i,v_j]\in E(G)$.
	\item If $i>j>k>\ell$ and $[v_i,v_\ell],[v_j,v_\ell],[v_j,v_k]\in E(G)$ then
		$[v_j,v_\ell]\in E(G)$.
	\end{enumerate}
\item \mbox{}$\langle$p.~38$\rangle$ {\bf Lemma 2.10}: An ordering $v_1,\dots,v_n$ of $V(G)$
	is a strong perfect elimination order of $G$ if and only if for each $i,j,k$
	and $\ell$ with $i\geq j$, $k\geq \ell$ and $i\sim \ell$, $j\sim\ell$, $j\sim k$ we have $i\sim k$.
\item $\langle$The notation $A \subset B$ (for two sets $A,B$) permits $A=B$;
	the symbol $\subsetneq$ is used for a strict subset.$\rangle$
	
\item \mbox{}$\langle$p.~82$\rangle$ {\bf Definition:} The vertices $u$ and $v$ are \emph{compatible}
	in the graph $G$ if $N[u]\subset N[v]$ or $N[v]\subset N[u]$. $\langle \dots \rangle$

	{\bf Definition:} A vertex $v$ of a graph is called \emph{simple} in $G$ if
	the vertices in $N[v]$ are pairwise compatible, or, equivalently, 
	if $\{N[u]: u\sim v\}$ is linearly ordered by inclusion.
\item \mbox{}$\langle$p.~83$\rangle$ {\bf Theorem 3.2:}   $G$ is strongly chordal if
	and only if every induced subgraph of $G$ has a simple vertex.

\item \mbox{}$\langle$pp.~111-114$\rangle$ {\bf Theorem 3.10:} $\langle$\cite{Gav74}$\rangle$: 
	A graph is chordal if and only if
	it is the intersection graph of a collection of subtrees of a tree. 

	$\langle$Since later proofs use not only the result but details
	of the construction and its properties, I repeat the necessity-proof here as well.$\rangle$

	\begin{proof} Necessity: Suppose $G$ is a connected chordal graph.
	Let $v_1,\dots,v_n$ be a perfect elimination order of $G$.
	We define a graph $T_G$ as follows:
	\begin{itemize}
	\item[(a)] $V(T_G)=\{v_1,\dots,v_n\}$.
	\item[(b)] If $i>j$ then $[v_i,v_j]\in E(T_G)$ if and only if
		$i=\min\{ k: k>j \text{ and } [v_k,v_j]\in E(G)\}$.
	\end{itemize}
	{\bf Claim 1:}  $T_G$ is a tree and for every $j\neq n$, the unique path
	in $T_G$ from $v_j$ to $v_n$, $v_jv_{j_1}\dots v_{j_k}v_n$, satisfies $j<j_1<\dots <j_k$.

	{\bf Proof of Claim 1:} $\langle \dots \rangle$
%$T_G$ contains no cycles since each vertex $v_j$ is adjacent
%	to at most one vertex $v_i$ where $i>j$. In order to show that $T_G$ is a tree it
%	suffices to show that for each $j$ there is a path in $T_G$ from $v_j$ to $v_n$.
%	Given the definition of $T_G$, it is enough to show that if $j\neq n$ then
%	$\{k: k>j \text{ and }[v_k,v_j]\in E(G)\}\neq \emptyset$.\todo{check whether this makes sense, hard to read}
%	Suppose that, for some $j\neq n$, $v_j$ is adjacent only to $v_i$ with $i<j$.
%	Since $G$ is connected there is a shortest path from $v_j$ to $v_n$, say $v_{i_0}v_{i_1}\dots v_{i_\ell}$
%	where $i=i_0$ and $n=i_\ell$.   By assumption $i_1<i$.    Since $i_1<n$ also, there is some $t$
%	between 0 and $\ell$ such that $i_t<i_{t-1}$ and $i_t<i_{t+1}$.   Since $v_1,\dots,v_n$ is
%	a perfect elimination order of $G$, $[v_{i_{t-1}},v_{i_{t+1}}]\in E(G)$, contradicting
%	the fact that we chose a shortest path from $v_i$ to $v_n$.   Thus, $T_G$ is a tree
%	and, for each $j\neq n$, there is exactly one $i>j$ such that $[v_i,v_j]\in E(T_G)$.
%	Consequently, for each $j\neq n$, the unique path in $T_G$ from $v_1$ to $v_n$,
%	$v_{j}v_{j_1}\dots v_{j_k}v_n$, satisfies $j<j_1<\dots<j_k$.   This completes the
%	proof of Claim 1. 

	\medskip
	For each $i$, let $P_i$ be the unique path in $T_G$ from $v_i$ to $v_n$,
	and let $T_i$ be the subgraph of $T_G$ induced
	by $\{v_k: k\leq i \text{ and } v_k\in N_G[v_i]\}$.

	\medskip

	{\bf Claim 2:} $T_i$ is connected for each $i$.

	{\bf Proof of claim 2:} Note that if $v_i$ lies on $P_k$ then $P_i\subseteq P_k$.
	Thus it suffices to show that if $v_k\in V(T_i)$ then $v_i$ lies on $P_k$ and
	$P_k-P_i$ is a subgraph of $T_i$.   We prove this by induction on $i-k$.   (Note
	that by the definition of $T_i$, $i-k\geq 0$ for all $v_k\in V(T_i)$.)   The case
	$i=k$ is trivial.   Suppose that $v_k\in V(T_i)-\{v_i\}$.   Then $k<i$.   Let $j$
	be the vertex of $P_k$ satisfying $[v_k,v_j]\in E(T)$. $\langle\text{sic}\rangle$  Then $[v_k,v_j]\in E(G)$
	and, moreover, $j>k$ by claim 1.     Hence $j>k$, $i>k$, and $[v_k,v_j],[v_i,v_k]\in E(G)$,
	whence $v_i\sim v_j$ in $G$ since $v_1,v_2,\dots,v_n$ is a perfect elimination order of $G$.
	Also $j\leq i$, by the definition of $T_G$ and the fact that $i>k$.   Thus $v_j\in V(T_i)$
	and $i-j<i-k$.   By the inductive hypothesis, $v_i$ lies on $P_j$ and $P_j-P_i$ is a 
	subgraph of $T_i$.   Thus $v_i$ lies on $P_k$ and $P_k-P_i$ is a subgraph of $T_i$,
	since $P_j=P_k-v_k$ and $v_k\in T(v_i)$.   The validity of claim 2 follows by induction.

	\medskip

	{\bf Claim 3:} For each $i$ and $j$, $v_i\sim v_j$ in $G$ if and only if $V(T_i)\cap V(T_j)\neq \emptyset$.

	{\bf Proof of claim 3:}  Suppose $[v_i,v_j]\in E(G)$.   Then $i\neq j$.   If $i<j$ then $v_i\in V(T_i)\cap V(T_j)$,
	by the definition of $T_i$ and $T_j$.   If $j<i$, then $v_j\in V(T_i)\cap V(T_j)$.

	Suppose that $i\neq j$ and $V(T_i)\cap V(T_j)\neq \emptyset$.   Let $v_k\in V(T_i)\cap V(T_j)$.   If
	$k=i$ or $k=j$ then $[v_i,v_j]\in E(G)$ by the definition of $T_i$ and $T_j$.   Otherwise $i\geq k$, $j\geq k$,
	and $[v_i,v_k],[v_j,v_k]\in E(G)$, whence $[v_i,v_j]\in E(G)$, since $v_1,v_2,\dots,v_n$ is a perfect
	elimination ordering of $G$, proving claim 3.

	\medskip

	Thus $G$ is the intersection graph of $\{T_1,\dots,T_n\}$, proving necessity in the case that
	$G$ is connected.   If $G$ is disconnected then we construct a tree $T_{G_i}$, and a collection
	of subtrees of $T_{G_i}$, as above, for each component $G_i$ of $G$.   We then join $v_n$ to
	the maximum vertex in the perfect elimination ordering of each component not containing $v_n$.
	The resulting tree will be denoted $T_G$.   It follows that $G$ is the intersection graph
	of the subtrees of $T_G$ chosen for each component.

	\bigskip
	Sufficiency: $\langle\dots\rangle$
	\end{proof}
\end{itemize}

\subsection{The actual proof}

$\langle$pp.~114-116$\rangle$
Several definitions are needed in order to give an intersection
graph characterization of the class of strongly chordal graphs.
Let $T$ be a tree with a distinguished vertex $r$ called the
root of $T$.   
Suppose that $T$ is weighted, i.e., that each
edges of $T$ is assigned a positive number called the length
of the edge.

For two nodes $u,v\in V(T)$ the \emph{weighted distance}
from $u$ to $v$, denoted $d_T(u,v)$, is the sum of the lengths
of the edges on the unique path from $u$ to $v$ in $T$.    If
$T_1$ and $T_2$ are two subtrees of $T$, then $T_1$ is \emph{full}
with respect to $T_2$, written $T_1\full T_2$, 
if for any two vertices $u,v$ of $T_2$ such that $d_T(r,u)\leq d_T(r,v)$, 
$v\in V(T_1)$ implies that $u\in V(T_1)$.

A collection $\{T_1,\dots,T_n\}$ of subtrees of $T$ is \emph{compatible}
if for each $i$ and $j$ either $T_i\full T_j$ or $T_j\full T_i$.

\begin{theorem}
A graph is strongly chordal if and only if it is the intersection
graph of a compatible collection of subtrees of a rooted weighted tree.
\end{theorem}
\begin{proof} (Necessity)   Suppose $G$ is strongly chordal.
Let $v_1,\dots,v_n$ be a strong elimination order of $G$.
Construct $T_G,T_1,\dots,T_n$ as in Theorem 3.10.
Root $T_G$ at $v_n$ and assign lengths to the edges of $T_G$ so
that $d_{T_G}(v_n,v_j)=n-j$.
In view of Theorem 3.10, it suffices to show that if $i>j$ then $T_i\full T_j$.
Let $i>j$. If $V(T_i)\cap V(T_j)=\emptyset$ then $T_j\full T_i$ trivially.
Otherwise, let $\ell=\min\{s: v_s\in V(T_i)\cap V(T_j)\}$.    Suppose that
$v_k\in V(T_j)$ and $d_{T_G}(v_n,v_k)\leq d_{T_G}(v_n,v_\ell)$, i.e.,
$k\leq \ell$.    
Then, $i>j$, $k\geq \ell$, and $(v_i,v_\ell),(v_\ell,v_j),(v_j,v_k)\in E(G)$,
whence $v_i\sim v_k$ in $G$, by Lemma 2.10
and the fact that $v_1,\dots,v_n$ is a strong elimination order of $G$.
Moreover, since $v_k\in T(v_j)$, we have $k\leq j$ and hence $k\leq i$.
Thus $k\in V(T_i)$ by the choice of $T_i$ and hence $T_j\full T_i$ by
definition.

\medskip

(Sufficiency) Let $T$ be a weighted tree with root $r$, and let
$\Lambda$ be a compatible collection of subtrees of $T$.   Suppose
that $G$ is the intersection graph of $\Lambda$.    Clearly every
induced subgraph of $G$ is the intersection graph of a compatible
collection of subtrees  of $T$.    Thus, by Theorem 3.2
it suffices to show that $G$ has a simple vertex.

Let $V(G)=\{v_1,\dots,v_n\}$ and let $T_1,\dots,T_n$ be
subtrees of $T$ in $\Lambda$ satisfying $v_i\sim v_j$ in $G$ if and
only if $V(T_i)\cap V(T_j)\neq \emptyset$.    For each $i$, let
$w_i$ be the unique vertex of $T_i$ of minimum weighted distance
to $r$.    We may assume, without loss of generality, that
$d_T(w_1,r)\geq d_T(w_j,r)$ for $j=2,\dots,n$.
We claim that $v_1$ is simple in $G$.   Suppose that $v_i,v_j\in N_G(v_1)$
and $T_i\full T_j$.   It suffices to show that $N_G(v_j)\subseteq N_G(v_i)$.
Since
$$ d_T(w_1,r)\geq d_T(w_j,r)$$
and
$$ V(T_1)\cap V(T_j)\neq \emptyset$$
it follows that
$$ w_1 \in V(T_j).$$
Similarly
$$ w_1 \in V(T_i).$$
Since $T_i> T_j$ we have
$$ w_j\in V(T_i).$$
Let $v_k\in N_G(v_j)$.   Then
$$ V(T_j)\cap V(T_k)\neq \emptyset.$$
If $w_j\in V(T_k)$ then $V(T_i)\cap V(T_k)\neq \emptyset$.  Otherwise
	$$w_k\in V(T_j)$$
	$$d_T(r,w_k)\leq d_T(r,w_1)$$
and 
	$$w_1 \in V(T_i)\cap V(T_j)$$
whence
	$$w_k\in V(T_i)$$
since $T_i\full T_j$.   
In either event
	$$V(T_i)\cap V(T_k)\neq \emptyset$$
and so
	$$v_k\in N_G(v_i).$$
Since $v_k$ was an arbitrary vertex of $N_G(v_j)$, it follows that
	$$N_G(v_j) \subseteq N_G(v_i)$$.
\end{proof}

\end{appendix}

\end{document}